\documentclass[sn-mathphys]{sn-jnl}

\jyear{2021}%

\theoremstyle{thmstyleone}%
\newtheorem{theorem}{Theorem}
\newtheorem{proposition}[theorem]{Proposition}%

\theoremstyle{thmstyletwo}%
\newtheorem{example}{Example}%
\newtheorem{conjecture}{Conjecture}%
\newtheorem{lemma}{Lemma}%

\theoremstyle{thmstylethree}%
\newtheorem{definition}{Definition}%
\newtheorem{notation}{Notation}%
\newcommand{\bco}{\begin{conjecture}}
\newcommand{\eco}{\end{conjecture}}
\newcommand{\bthe}{\begin{theorem}}
\newcommand{\ethe}{\end{theorem}}
\newcommand{\ble}{\begin{lemma}}
\newcommand{\ele}{\end{lemma}}
\newcommand{\bde}{\begin{definition}}
\newcommand{\ede}{\end{definition}}

\newcommand{\bno}{\begin{notation}}
\newcommand{\eno}{\end{notation}}
\newcommand{\beq}{\begin{equation}}
\newcommand{\eeq}{\end{equation}}

\begin{document}

\title[Some discrete subgroups of $SL(3,\mathbb{C})$]{Discrete subgroups generated by lattices in opposite horospherical subgroups of $SL(3,\mathbb{C})$ following Hee Oh}


\author*{\fnm{Eduardo} \sur{Montiel}}\email{eduardo.montiel@im.unam.mx}


\affil{\orgdiv{Instituto de Matem\'aticas}, \orgname{UNAM}, \orgaddress{\street{Av. Universidad}, \city{Cuernavaca}, \postcode{62210}, \state{Morelos}, \country{M\'exico}}}




\abstract{We prove that a discrete subgroup generated by two lattices in opposite minimal horospherical subgroups of SL(3,$\mathbb{C}$) is arithmetic and thus by a Borel and Harish-Chandra also a lattice. We follow the method and ideas used by Oh in \cite{Ref13}. There, Oh proves the same result for lattices in minimal horospherical subgroups of SL(n,$\mathbb{R}$).

The method consists of studying the orbits of the generating lattices in the horospherical subgroups under the conjugation action of the commutator of the corresponding Levi subgroup in order to obtain a rational structure for SL(3,$\mathbb{C}$). 

A theorem of Ratner reduces the possibilities for the closure of the orbits. Then, by the discreteness of the generated subgroup, it can be shown the closeness of the orbits.
 
Finally, using this information, it's possible to find a rational form such that the subgroup generated by the lattices is arithmetic.}




\maketitle
\section{Introduction}\label{sec1}
In the 70s, Gregory Margulis in \cite{Ref11} proved his celebrated theorem of arithmeticity:

\bthe [Margulis' Arithmeticity Theorem]
Let $G$ be a connected semisimple algebraic $\mathbb{R}$-group without compact factors and with trivial center, of $\mathbb{R}$-rank greater than 1, and let $\Gamma$ be an irreducible non-uniform lattice in $G_\mathbb{R}$. Then $\Gamma$ is an arithmetic subgroup of $G$. More accurately, we can give to $G$ a   $\mathbb{Q}$-structure such that the groups $\Gamma$ and $G_\mathbb{Z}$ are commensurable. 
\ethe

This is a partial converse of a Borel and Harish-Chandra theorem, see \cite{Ref2},  that asserts if $G$ is an algebraic semisimple Lie group defined over $\mathbb{Q}$, then $G_\mathbb{Z}$ is a lattice of  $G_\mathbb{R}$. 

An important ingredient in the Margulis proof is the existence of lattices in opposite horospherical subgroups, the converse of the last fact was also conjectured by Margulis after hears a Selberg lecture at Yale in 1992 in the following form:

\bco
Let $G$ be an $\mathbb{R}$-group semisimple of real rank at least two, $U_1$ and $U_2$ a pair of opposite horospherical subgroups of $G(\mathbb{R})$, and $F_i$ a lattice of $U_i$ such that for any proper normal subgroup $H$ of $G(\mathbb{R})$, $H \cap F_{i}$ is finite for $i=1,2$. Then, if the subgroup $\Gamma_{F_{1},F_{2}}$ generated by $F_{1}$ and $F_{2}$ is discrete, there exists a $\mathbb{Q}$-form of $G$ such that $U_i$ is defined over $\mathbb{Q}$ and $F_{i}$ is commensurable to $U{_i(\mathbb{Z})}$  for each $i=1,2$ and  $\Gamma_{F_{1},F_{2}}$ is an arithmetic subgroup.   
\eco

Hee Oh was a pioneer in studying this conjecture, in her first approximation, see \cite{Ref13}, she solved the problem positively for minimal opposite horospherical subgroups of $SL(n,\mathbb{R})$, and later, for many cases in \cite{Ref14}. One of the missing cases in this last approach occurs when the horospherical subgroups are Heisenberg subgroups of $SL(3,\mathbb{R})$, but it was proved in \cite{Ref1} by Oh and Benoist using ideas of Selberg.

In this paper, we follow the work of Oh in \cite{Ref13} to prove the following result:

\bthe
Let $F_{1}$ and $F_{2}$ be lattices in minimal opposite horospherical subgroups of $SL(3,\mathbb{C})$, if $<F_1,F_2>$, the subgroup generated by $F_1$ and $F_2$, is discrete then  is arithmetic.  
\ethe

A pair of minimal horospherical subgroups of $SL(3,\mathbb{C})$ is conjugate to either: 

\begin{equation} \notag \left\{ U=\left\{ \left(\begin{array}{ccc}
1 & * & *\\
0 & 1 & 0\\
0 & 0 & 1
\end{array}\right)\right\} ,U^{-}=\left\{ \left(\begin{array}{ccc}
1 & 0 & 0\\
* & 1 & 0\\
* & 0 & 1
\end{array}\right)\right\} \right\} 
\end{equation}
 
or 

\begin{equation} \notag
  \left\{ U=\left\{ \left(\begin{array}{ccc}
1 & 0 & * \\
0 & 1 & *\\
0 & 0 & 1
\end{array}\right)\right\} ,  U^{-}=\left\{ \left(\begin{array}{ccc}
1 & 0 & 0\\
0 & 1 & 0\\
* & * & 1
\end{array}\right)\right\} \right\} 
 \end{equation}

So, the result in this work keeps some similarity with the combination theorems because it gives a criterion for discreetness of the generated subgroup of two lattices. Indeed the study of discrete free subgroups generated by a pair of non-commutative parabolic subgroups of $PSL(2,\mathbb{C})$ is the one that maintains the most similarities with the result of Oh, we recall that study; up to conjugation we can always write a pair of non-commutative parabolic subgroups of $PSL(2,\mathbb{C})$ like 

\begin{equation}\notag 
X=\left(\begin{array}{cc}
1 & 1\\
0 & 1
\end{array}\right)
 , Y_{\rho}=\left(\begin{array}{cc}
1 & 0\\
\rho & 1
\end{array}\right)
\end{equation}

for some $\rho \in \mathbb{C} \setminus \{0\}$. Then we get a parametric space obtained by identifying each $\rho$ with the class of the subgroup $\Gamma_{\rho}=<X,Y_{\rho}>$. The set of parameters $\rho$ for which $\Gamma_{\rho}$ is discrete and free, is known as the Riley slice. 

Some experts in the study of this set are J. Parker, B. Maskit, L. Keen, C.Series and Y. Komori. 

Parker and Will in \cite{RefPW} study a set analogous of the Riley slice. Specifically, they study subgroups of $PU(2,1)$ generated by a pair of non-commutative unipotent  transformations $A,B$ such that the product $AB$ is also unipotent. As in the classic case, they define a set called the complex hyperbolic Riley slice, that consist of classes of conjugation for which the generated subgroup is discrete and free.  

In section 2, we give a lot of preliminary material with the intention to make the  article self-contained and to exhibit classical results in the theory of discrete groups. That section contains, for example, a beautiful theorem of Ratner with a dynamical flavor.  

In section 3, we present the proof of the arithmeticity of the group generated by  $F_1$ and $F_2$. The idea is to obtain a rational structure in $G$ that fits with the rational structures in the opposite horospherical groups  that arises naturally of the lattices  $F_1$ and $F_2$.

\section{Preliminaries}\label{sec2}

We recall some definitions, results and notation.  

\subsection{Algebraic groups}
\label{subsec1}

The main references for this subsection are \cite{RefBT} and \cite{RefHL}. 

\bde 
A subgroup $G \subset GL(n,\mathbb{C})$ is a linear algebraic group if there is a set $\mathcal{A}$ of polynomials on $M(n,\mathbb{C})$ so that 
 \begin{equation}\notag
 G=\{g \in GL(n,\mathbb{C}) : f(g)=0, \forall f\in  \mathcal{A}\}.
 \end{equation}
\ede

 We recall the concepts of $\mathbb{Q}$-groups and $\mathbb{Q}$-forms.

\bde See \cite[p.~217]{RefHL} or \cite[p.~57]{RefBT}. 
 Let $G$ be an algebraic subgroup of $GL(n,\mathbb{C})$ and $k$ a subfield of  $\mathbb{C}$, we say that $G$ is defined over $k$  (or that $H$ is a $k$-group) if $G$ consists of all matrices whose entries annihilate some set of polynomials on $M(n,\mathbb{C})$ with coefficients in $k$. 
\ede
  
\bde See \cite[p.~106]{RefBT} or \cite[p.~625]{Ref14}. Let $G$ be a $\mathbb{R}$-group and $k$ a subfield of $\mathbb{R}$. A $k$-form of $G$ is a pair $(\tilde{G},f)$  where $\tilde{G}$ is an algebraic group defined over $k$ and $f:G \rightarrow \tilde{G}$ is an isomorphism defined over $\mathbb{R}$ \footnote{An isomorphism $f$ is defined over $\mathbb{R}$ if the coordinate functions of $f$ all lie in $\mathbb{R}[T_1,..T_n]$, where $n$ is the dimension of $\tilde{G}$ as variety, see \cite[p.~218]{RefHL}.}.   
\ede

\bno See \cite[p.~625]{Ref14} or \cite[p.~376]{Ref10}
For $G$ an algebraic subgroup of \;$GL(n,\mathbb{C})$, $J$ a subring of\; $\mathbb{C}$ and $k$ a subfield of\; $\mathbb{C}:$ 

\begin{enumerate}
	\item $G(J)$ denotes the elements of $G$ with matrix coefficients belonging to $J$ whose determinant is a unit of $J$;
	\item If $G$ is defined over $\mathbb{R}$ and has a $k$-form $(\tilde{G},f)$, we denote $f(\tilde{G}(J))$ by $G(J)$;  
	
  \item If an algebraic $\mathbb{R}$-subgroup $H$ of $G$ is such that $f^{-1}(H)$ is a $k$-subgroup of $\tilde{G}$, we say that $H$ is defined over $k$ and we denote $H \cap G(J)$ by $H(J)$.
\end{enumerate}
\eno

\begin{theorem} See \cite[p.~36]{RefZE} 
Let $k$ be a subfield of $\mathbb{C}$ and $G \subset GL(n,\mathbb{C})$ an algebraic group. If $G \cap GL(n,k)$ is dense Zariski in $G$, then $G$ is defined over $k$.  
\end{theorem}

\bde  \cite[p.~134,135]{RefHL} Let $G$ be an algebraic group.
\begin{enumerate}
	\item  If $B$ is a maximal closed connected solvable subgroup, $B$ is called a Borel subgroup.
	\item  If $P$ is a closed subgroup such that $P$ contains a Borel subgroup, then $P$ is called a parabolic subgroup of $G$.
\end{enumerate}
  
\ede

A Borel subgroup $B$ is a parabolic subgroup, and we often refer to it as a minimal parabolic subgroup. \\

By the Lie-Kolchin theorem we known that the only (up to conjugacy) parabolic subgroups of $G=SL(3,\mathbb{C})$ are $G$, $B$ the upper triangular group of $G$, and the groups 
\begin{equation}\notag
\left\{ \left(\begin{array}{ccc}
* & * & *\\
0 & * & *\\
0 & * & *
\end{array}\right)\right\} \; , \; \left\{ \left(\begin{array}{ccc}
* & * & *\\
* & * & *\\
0 & 0 & *
\end{array}\right)\right\}.
\end{equation}

\begin{definition} See \cite[p.~125]{RefHL} or \cite[p.~59]{RefBT}
Let $G$ be an algebraic group . We define: 
\begin{description}	

\item[(a)] the \textbf{radical} $R(G)$ of $G$ as the largest connected normal solvable subgroup  (closed) of $G$.
\item[(b)] the \textbf{unipotent radical } $R_u(G)$ of $G$  as the largest connected normal   unipotent (closed) subgroup of $G$.
\item[(c)] A group $G$ is called \textbf{semisimple }(resp. \textbf{reductive}) if $R(G)=\{e\}$ (resp. $R_u(G)=\{e\}$) .
\end{description}
\end{definition}

For example the group $GL(n,\mathbb{C})$ is reductive and $SL(n,\mathbb{C})$ is semisimple.

\begin{theorem} [Levi decomposition] \cite[p.~185]{RefHL}  or \cite[p.~57]{RefBT}
Let $G$ be a reductive group and $P$ a parabolic subgroup of $G$. Then $P=L \ltimes V$, where $V=R_u(P)$. If\: $L^{\prime}$ is another subgroup such that $P=L^{\prime} \ltimes V$ then $L^{\prime}$ is conjugate to $L$ by an element of $V$. We call $L$ a Levi subgroup.  
\end{theorem}

\begin{definition} \cite[p.~88]{RefBT}
Let $G$ be a reductive group. Two parabolic subgroups $P^+$ y $P^-$ are called opposite if their intersection  is a Levi subgroup of both, i.e. if $P^+=L \ltimes V^+$ and $P^-=L \ltimes V^-$ where $L=P^+ \cap P^-$.
\end{definition}

\bde \cite[p.~625]{Ref14}
Let $P$ be a parabolic subgroup of $SL(n,\mathbb{C})$, the subgroup $R_u(P)$ is called a horospherical group.  
\ede 
 
 We call a parabolic group maximal if it is the greatest group among all parabolic groups distinct of the whole group. A horospherical group is called minimal if it is the unipotent radical of a maximal parabolic subgroup. \\
 
  In $SL(3,\mathbb{C})$ every pair composed by two opposite minimal horospherical  subgroups is conjugate to a pair $U$, $U^-$ with \begin{equation}\notag U=\left\{ \left(\begin{array}{ccc}
I_m & M_{m \times k }(\mathbb{C}) \\
0 & I_k \\
\end{array}\right)\right\} \quad\mathrm{and} \quad
U^{-}=\left\{ \left(\begin{array}{ccc}
I_m & 0 \\
M_{k \times m} (\mathbb{C})  & I_k\\
\end{array}\right)\right\} 
 \end{equation} 

for some $k,m \in \{1,2\}$ such that $k+m=3$.

\subsection{Lattices and arithmetic groups}
\label{subsec2}

For this section we has taken as references \cite{RefMORR} and \cite{RefBI}  

\bde \cite[p.~21]{RefRD}
Let $G$ be a Lie group and $\Gamma$ a discrete subgroup of $G$. We say that $\Gamma$ is a lattice in $G$ if $G/ \Gamma$ has a finite volume. In particular, if $G/\Gamma$ is compact we say that $\Gamma$ is a cocompact lattice or a uniform lattice in $G$.   
\ede 

For example, $\mathbb{Z}^{n}$, $SL(n,\mathbb{Z})$ and $SL(n,\mathbb{Z}[i])$ are lattices in $\mathbb{R}^{n}$, $SL(n,\mathbb{R})$ and $SL(n,\mathbb{C})$, respectively.

We recall the Iwasawa decomposition for $SL(n,\mathbb{R})$ and $SL(n,\mathbb{C})$. 

\begin{theorem} [Iwasawa's decomposition] 
Let $G=SL(n,\mathbb{R})$ (resp. $G=SL(n,\mathbb{C})$). Let $K=SO(n)$ (resp. $K=SU(n)$ ), $A$ be the group that consists of all diagonal matrices with positive eigenvalues and determinant 1, and $N$ the group that consists of all the unipotent upper triangular matrices.   
The map $(k,a,n) \mapsto k \cdot a \cdot n$ is a diffeomorphism of $K \times A \times N$ onto $G$. 
\end{theorem}

Let $A,N$ as above, given two positive numbers $t,u$, we construct the sets:

\begin{equation}\notag
A_t=\{ a\in A^* : a_{ii} \leq t \cdot a_{jj} \quad (i=1, \ldots n-1) \} 
\end{equation} 
\begin{equation}\notag
N_u=\{ n \in N : \lvert n_{ij} \rvert \leq u  \quad (1 \leq i<j \leq n) \}
\end{equation}
 
 We call Siegel set of $G$ to a subset of the form $\mathfrak{G}_{t,u}=K \cdot A_t \cdot N_u$. The Siegel sets have finite Haar measure, for a proof in the real case see \cite[p.~17]{RefBI} and \cite[p.~17]{Ref15} for the complex case.  This fact and the next theorem show that $SL(n,\mathbb{Z})$ is a lattice in $SL(n,\mathbb{R})$. 
 
 \begin{theorem}\cite[p.~13]{RefBI}
 $SL(n,\mathbb{R})=\mathfrak{G}_{\frac{2}{\sqrt{3}},\frac{1}{2}} \cdot SL(n,\mathbb{Z})$.
 \end{theorem}

We denote by $\mathcal{R}$ the set of lattices in $\mathbb{R}^n$, this set is identified with $GL(n,\mathbb{R})/GL(n,\mathbb{Z})$. We define $\triangle$ as the map that sends a lattice $\Gamma$ to the volume of the quotient $\mathbb{R}^n / \Gamma$, more explicitly, if $\Gamma=g(\Gamma_0)$ for some representative $g\in GL(n,\mathbb{R})$ and $\Gamma_0=\mathbb{Z}^n$, then $\triangle(\Gamma)= \lvert \det g \rvert$. 
  
\begin{theorem}[Mahler's compactness criterion]\cite[p.~16]{RefBI}
Let $M \subset \mathcal{R}$ be a subset of the space of lattices. Then $M$ is relatively compact if and only if: 

\begin{description}
\item[(a)] $\triangle$ is bounded in  $M$ and 
\item[(b)] there exists a neighborhood \;$U$ of \;0 in $\mathbb{R}^n$ such that $\Gamma \cap U=\{0\}$ for every  $\Gamma \in M$. 
\end{description}
 
\end{theorem}

\begin{theorem}[Hermite] \cite[p.~208]{RefCA} or \cite[p.~39]{RefMPE}     
Given a matrix $g \in GL(n,\mathbb{R})$, we have

\begin{equation} \notag
\underset{x\in\mathbb{Z}^{n}\smallsetminus\{0\}}{\min}\left\Vert g \cdot x\right\Vert \leq(2/\sqrt{3})^{(n-1)/2}\lvert \det g \rvert ^{1/n}
  \end{equation}

\end{theorem}

\ble
Let $\mathcal{R}$ be the space of lattices in $\mathbb{R}^{n}$, and $M \subset \mathcal{R}$ relatively compact. Then there is a positive number $d$ such that every lattice in $M$ has a basis composed of elements with norm less than $d$.  
\ele

\begin{proof}
Let $\varphi:GL(n,\mathbb{R}) \rightarrow \mathcal{R}$ be the map defined by $\varphi(g)= g(\Gamma_0)$, where $\Gamma_0 = \mathbb{Z}^{n}$. Take $\mathfrak{S}$ a Siegel set in $GL(n,\mathbb{R})$ such that the image of $\mathfrak{S}$ under $\varphi$ is $\mathcal{R}$.  The relative compactness of $M$ is equivalent to find a subset $M^{\prime} \subset \mathfrak{S}$ relatively compact such that $\varphi(M^{\prime})=M$. Let $g=kan \in M^{\prime}$, as the subset $M^{\prime}$ is relatively compact then there are two constants $\alpha, \beta > 0$  such that $\alpha \leq (a_g)_{ii} \leq \beta$ for all $g \in M^{\prime}, i=1,...,n$, also there exist $u>0$ such that $\lvert (n_{ij})\rvert \leq u$. Note that every column vector of the $an$-part of $g$ has norm less than $\beta ((n-1)u+1)$.

As $k$ is an orthogonal matrix every row vector has norm 1,  then $\lvert g_{ij} \rvert$ is less than $\beta ((n-1)u+1)$, so every vector of $g$ has norm less than $\beta n((n-1)u+1)$. Then every lattice in $M$ has a basis composed by vectors with norm less than $d$, with $d\geq \beta n((n-1)u+1)$. 
\end{proof}

\ble \label{bs}
Let $M$ be a  relatively compact subset of $\mathcal{R}$ the space of lattices in $\mathbb{R}^{n}$. There are positive constants $d_1$ and $d_2$ such that for any lattice $L \in M$ exists a basis composed by elements of norm less than $d_1$ and grater than $d_2$. 
\ele

\bde \cite[p.~133]{RefMD}
Let $G$ be a connected semisimple group without compact factors. A lattice $\Gamma$ is said reducible  if  $G$ admits infinite connected normal subgroups $H$ and $H^{\prime}$ such that $HH^{\prime}=G$, $H\cap H^{\prime}$ is discrete and $\Gamma/(\Gamma \cap H)\cdot (\Gamma \cap H^{\prime})$ is finite.  A lattice is irreducible if it is not irreducible.
\ede

\begin{example}
 The lattice obtained by identifying the subgroup $SL(2,\mathbb{Z}[\sqrt{2}])$ in \newline $SL(2,\mathbb{R}) \times SL(2,\mathbb{R})$ via the map $g \mapsto (g,\sigma(g))$ where $\sigma$ is the natural extension of the Galois embedding that sends the number $a+b \sqrt{2}$ to $a-b \sqrt{2}$ is irreducible, unlike $SL(2,\mathbb{Z}) \times SL(2,\mathbb{Z})$ that is reducible. \\
 \end{example}
 
\begin{definition}
Let $G,H$ be algebraic groups and $f:G \rightarrow H$ a surjective homomorphism. We say that $f$ is an isogeny if $ker(f)$ is finite.   
\end{definition}

\begin{example}
There exists an isogeny of $ SL(2,\mathbb{C})$ to $ SO(3,\mathbb{C})$. Consider the space $V$ of $2 \times 2$ complex matrices with trace $0$, and the symmetric bilinear form $<x,y>=\mathrm{tr}(xy)$. The group $ SL(2,\mathbb{C})$ acts on $V$ by $g \cdot x = gxg^{-1}$, moreover this action preserves $<,>$. \\ Let $E_1=\left(\begin{array}{cc}
1 & 0\\
0 & -1
\end{array}\right)$, $E_2=\left(\begin{array}{cc}
0 & 1\\
1 & 0
\end{array}\right)$, $E_3=\left(\begin{array}{cc}
0 & 1\\
-1 & 0
\end{array}\right)$ be an orthogonal basis of $V$. \\
As $<E_1,E_1>=2$,$<E_2,E_2>=2$ and $<E_3,E_3>=-2$ the bilinear form is non-degenerate. So, $ SL(2,\mathbb{C})$ maps to a copy of $ SO(3,\mathbb{C})$, and the kernel is $\{ \pm I_2 \}$. 
\end{example}

\bde
Let $G$ be a linear algebraic group defined over $\mathbb{R}$ and $(\tilde{G},f)$ a $\mathbb{Q}$-form of $G$. A subgroup $H$ commensurable to $G(\mathbb{Z})$ is called an arithmetic subgroup of  $G(\mathbb{R})$. 
\ede

\begin{example}
The following are examples of arithmetic subgroups:
\begin{enumerate}
	\item The group $SL(n,\mathbb{Z})$ in $SL(n,\mathbb{R})$;
	\item the group $SO(p,q)\cap SL(n,\mathbb{Z})$ in $SO(p,q)$;
	\item consider $\sigma$ the automorphism of order 2 of $\mathbb{Q}[\sqrt{2}]$. The group \newline $\Gamma=\{(g,g^{\sigma}) : g\in SL(n,\mathbb{Z}[\sqrt{2}]) \}$ is an arithmetic subgroup of $SL(n,\mathbb{R}) \times SL(n,\mathbb{R})$; 
	\item   The group $SL(n,\mathbb{Z}[i])$ in  $SL(n,\mathbb{C})$.
\end{enumerate}
\end{example}

We recall a classical result of Borel, for a proof see \cite{RefZE}. 

\bthe[Borel density theorem]
Let $G$ be a semisimple Lie group without compact factors, $\Gamma$ a lattice in $G$ and a finite-dimensional representation $\rho: G \rightarrow GL(n,\mathbb{R}) $ of $G$. Then $\rho(\Gamma)$ is Zariski dense in $\rho(G)$.
\ethe

This theorem tell us, for example, that if $x_{ij}:SL(n,\mathbb{R}) \rightarrow \mathbb{R}$ is the map that returns the matrix coefficient at position $(i,j)$, then a polynomial in the $x_{ij}$ variables  with real coefficients that vanishes on $SL(n,\mathbb{Z})$ must vanishes on  $SL(n,\mathbb{R})$.\\

We recall an important theorem of Borel and Harish-Chandra, see \cite{Ref2},  that allows finding lattices in every semisimple linear algebraic group. 

\bthe[Borel and Harish-Chandra]
Let $G$ be a semisimple linear algebraic group defined over $\mathbb{Q}$. Then $G_\mathbb{Z}$ is a lattice in $G_\mathbb{R}$. 

\ethe

\subsection{Useful results on discrete subgroups}
\label{subsec3}

\bde
Let $G$ be a locally compact group and $S_n$ a sequence of subsets of $G$. We say that $S_n$ converges to $S$ if for every compact subset $K\subset G$ and a neighborhood $U$ of $e$ in $G$, there is an integer $r=(K,U)$ such that for all $n \geq r$ and $x \in S_n \cap K$, $xU \cap S \neq \emptyset$ and for all $y \in S \cap K$, $yU \cap S_n \neq \emptyset$. 
\ede

\begin{theorem}[Chabauty] \label{cha} \cite[p.~26]{RefRD}
Let $G$ be a Lie group and $\Gamma_n$ a sequence of lattices in $G$ such that for some open set $W$ of $G$ with $e \in W$, $W \cap \Gamma_n =\{e\}$ for all $n$. Then a subsequence $\Gamma_{i_{n}}$  of $\Gamma_n$ converges to $\Gamma$ and $\Gamma$ is a discrete subgroup. Furthermore, if $\mu$ is a right Haar measure on $G$, $\mu(G/ \Gamma) \leq \lim \inf \mu(G / \Gamma_{i_{n}})$
 \end{theorem}

\bde
Let $G$ be a lie group and $H$ a subgroup of $G$. We say that that $H$ has the property (P) if every Ad$(H)$-stable subspace of $\mathfrak{g}_{\mathbb{C}}$ is  Ad$(G)$-stable, where Ad is the adjoint representation of $G$ in the complexification $\mathfrak{g}_{\mathbb{C}}$ of the Lie algebra $\mathfrak{g}$.
\ede

Every Zariski dense subgroup of an algebraic group has the property (P). 

\begin{theorem} [Wang] \label{wa} \cite[p.~154]{RefRD}
Let $G$ be a semisimple Lie group and $K$ a compact set of $G$. Then there exists a neighborhood $V$ of $e$ en $G$ such that the following holds:  if $\Gamma$ is a discrete subgroup of $G$ such that $\Gamma \cap V$ generates a subgroup with property (P), then $\Gamma \cap V= \{e\}$.    
\end{theorem}

\begin{theorem}[Zassenhaus] \label{zama} \cite[p.~147]{RefRD}
let  $G$ be a Lie group. Then there exists a neighborhood $U$ of $e$ in $G$ such that if $\Gamma$ is any discrete group of $G$, $\Gamma \cap U$ is contained in a connected nilpotent Lie subgroup of $G$.   
\end{theorem} 

In fact, we can construct a neighborhood of the identity such that the commutator contracts that neighborhood.  We prove this fact in the particular case of $GL(n,\mathbb{C})$. For this purpose consider the norm $\left\| \cdot \right\|$ in $GL(n,\mathbb{C})$ given by the inner product $<z,w>=\sum_1^{n}z_i \overline{w_i}$ in $\mathbb{C}^{n^2}$.

\ble \label{za}
Let $G$ be $GL(n,\mathbb{C})$. There exists an $\epsilon >0$ such that if \:$V_\epsilon$ is the $\epsilon$-neighborhood of \:$I_n$ in $G$ then for any $A,B \in V_\epsilon$,  $\left\|[A,B]-I_n\right\| < \frac{1}{2} \mathrm{min}\{\left\|A-I_n\right\|,\left\|B-I_n\right\|\}$.  
\ele 

\begin{proof}
We define the function $\sigma$ given by $\sigma(C):=\max\{\left\|C-I_n\right\|,\left\|C^{-1}-I_n\right\|\}$, take $V$ an $\epsilon$-neigborhood such that $\sigma(C)<c=1/16$ for all $C \in V$. Suppose that $\sigma(A) \leq \sigma(B)$. By the submultiplicativity of the norm we have, 

\begin{equation}\notag
\left\|ABA^{-1}B^{-1}-I_n\right\| \leq \left\|A^{-1}\right\|\left\|B^{-1}\right\|\left\|AB-BA\right\|.
\end{equation}

Note that $\left\|C^{-1}\right\|<1+c$ when $\sigma(C)<c$, so 
\begin{equation}\notag
\left\|ABA^{-1}B^{-1}-I_n\right\|< (1+c)^2  \left\|AB-BA\right\|.
\end{equation} 
In the other hand, 
\begin{equation}\notag
\left\|AB-BA\right\|=\left\|(A-B)(A-I_n)-(A-I_n)(A-B)\right\| \leq
\end{equation}
\begin{equation}\notag
\leq \left\| (A-B)(A-I_n)\right\|+ \left\| (A-I_n)(A-B)\right\| \leq 2\sigma(A)\left\|A-B\right\| \leq
\end{equation}
\begin{equation}\notag
\leq 2\sigma(A)(\sigma(A)+\sigma(B)) \leq 4 \sigma(A) \sigma(B).
\end{equation}
By the two inequalities we have, 
\begin{equation}\notag
\left\|ABA^{-1}B^{-1}-I_n\right\| < 4c(1+c)^2\sigma(A), 
\end{equation}
by the election of $c$ we have that $4c(1+c)^2<\frac{1}{2}$, and then $\left\|ABA^{-1}B^{-1}-I_n\right\|< \frac{1}{2}\sigma(A)$. Thus, $\left\|ABA^{-1}B^{-1}-I_n\right\|< \frac{1}{2} \min \{\sigma(A),\sigma(B)\}$, interchanging the roles of $A$ and $B$ we obtain $\left\|(ABA^{-1}B^{-1})^{-1}-I_n\right\|< \frac{1}{2} \min \{\sigma(A),\sigma(B)\}$. Then, $\sigma([A,B]) < \frac{1}{2} \min \{\sigma(A),\sigma(B)\}$    

\end{proof} 

One of the most powerful and beautiful tools used by Oh is the Ratner's theorem, that reduces the possibilities for the orbits of certain actions. For the proof we invited the reader to the article of Ratner \cite{Ref16}.      

\begin{theorem} [Ratner] \label{ra}
Let $G$ be a real Lie group, $\Gamma$ a lattice in $G$, let $X$ be the homogeneous space $G/\Gamma$. Let $U$ be a connected subgroup of $G$ generated by one-parameter unipotent subgroups. Then for any $x\in X$, the closure of the orbit $Ux$ is itself a homogeneous space of finite volume; in other words, there exists a closed subgroup $U \leq H \leq G$ such that $\overline{Ux}=Hx$ and $x \Gamma x^{-1} \cap H$ is a lattice in $H$.  
\end{theorem}

An elemental matrix is a matrix with 1's in the diagonal, $(a_{i_0j_0})=r$ for a unique pair  $\{i_0,j_0\}$, $i_0 \neq j_0$ and $(a_{ij})=0$ for the remaining $\{i,j\}$. Let $K$ be a field, we denote by $E_n(K)$ the subgroup of $GL(n,k)$ generated by elemental matrices. It's possible to show that $E_n(K)=SL(n,K)$, so $SL(n,K)$ is generated by one-parameter unipotent subgroups. For a proof, we refer the reader to \cite{RefHOC}.\\ 

The following three theorems present the classification of the maximal connected Lie subgroups of $SL(n,\mathbb{C})$ which is the result of Dynkin's work, see \cite{RefGOV}.

\begin{theorem}
Let $H$ be a maximal connected Lie subgroup of $SL(n,\mathbb{C})$. If $H$ is reducible then $H$ is a maximal parabolic subgroup of $SL(n,\mathbb{C})$. Conversely, if $H$ is a maximal parabolic subgroup of $SL(n,\mathbb{C})$, then $H$ is a   maximal connected Lie subgroup.
\end{theorem}

\begin{theorem}
Let $H$ be a maximal connected Lie subgroup of $SL(n,\mathbb{C})$. If $H$ is irreducible then $H$ is conjugate to a subgroup  $SL(s,\mathbb{C}) \otimes SL(t,\mathbb{C})$, with $n=st$, $s,t \geq 2$. Conversely, if $H$ has the form $SL(s,\mathbb{C}) \otimes SL(t,\mathbb{C})$, with $n=st$, $s,t \geq 2$, then $H$ is a   maximal connected Lie subgroup.
\end{theorem}

\begin{theorem}\label{Max}
Let $R:H \rightarrow GL(V)$ be a non-trivial irreducible linear representation of a simply connected simple Lie group $H$. If there are no non-degenerate bilinear forms in $V$ invariant under $R$, then  $R(H)$ is a maximal connected subgroup of  $SL(V)$, and if $R$ is orthogonal or symplectic, then  $R$ is a maximal connected subgroup of $SO(V)$ or $SP(V)$, respectively. The only exceptions are the representations listed in Table 7 in \cite{RefGOV}.     
\end{theorem}

This result is proved by Oh in \cite{Ref14} based on the works \cite{RefRAN} and \cite{RefVO} of Raghunathan and Venkataramana respectively. 
\begin{theorem}\label{rv}
Let $G$ be a semisimple $\mathbb{Q}-$group of real rank at least two that does not contain normal non trivial $\mathbb{Q}$-subgroups. Let $F_1\subset U$, $F_2\subset U^{-}$ be a pair of lattices contained in a pair of opposite horospherical $\mathbb{Q}$-subgroups. If $F_1$ and $F_2$ are commensurable to $U(\mathbb{Z})$ and $U^{-}(\mathbb{Z})$ respectively, then $<F_1,F_2>$ is commensurable to $G(\mathbb{Q})$.  
\end{theorem}

\section{Arithmeticity of certain combined groups}
\label{sec3}

 In this section, we follow the ideas of Oh to prove the main theorem. The plan consist of study the orbits obtain by the conjugation action of the commutator of the Levi subgroup on the space of lattices of the horospherical subgroups.  We adapted the proofs of Oh to our case and give an amalgamated and perhaps clearer version of the method followed by her. 
 
 We use the norms $\left\| x \right\|$
and $\left\| g \right\|$ for $x \in \mathbb{C}^l$ and $g \in SL(l,\mathbb{C} )$ that correspond to the usual inner products of $\mathbb{C}^l$ and $\mathbb{C}^{l^2}$.
We remember that every pair $\{U$, $U^{-}\}$ of two opposite minimal horospherical subgroups in  $SL(3,\mathbb{C})$ is conjugate to either   

\begin{equation}\notag
\left\{ \left\{ \begin{array}{ccc}
1 & * & *\\
0 & 1 & 0\\
0 & 0 & 1
\end{array}\right\},\left\{ \begin{array}{ccc}
1 & 0 & 0\\
* & 1 & 0\\
* & 0 & 1
\end{array}\right\} \right\} 
 \end{equation} 

or 

\begin{equation}\notag
\left\{ \left\{ \begin{array}{ccc}
1 & 0 & *\\
0 & 1 & *\\
0 & 0 & 1
\end{array}\right\} ,\left\{ \begin{array}{ccc}
1 & 0 & 0\\
0 & 1 & 0\\
* & * & 1
\end{array}\right\} \right\} 
 \end{equation}

We work on the first pair, the second one is treated in the same way. Note that both $U$ and $U^{-}$ are isomorphic to $\mathbb{C}^2$, so we can identify the space of lattices in $U$ ($U^{-}$) with the homogeneous space $GL(4,\mathbb{R})/GL(4,\mathbb{Z})$. \\

The normalizer of $U$, $U^{-}$ is 

\begin{equation}\notag
\left\{ \left(\begin{array}{ccc}
* & * & *\\
0 & * & *\\
0 & * & *
\end{array}\right)\right\} ,\left\{ \left(\begin{array}{ccc}
* & 0 & 0\\
* & * & *\\
* & * & *
\end{array}\right)\right\} 
\end{equation} 

respectively, thus the Levi subgroup is 

\begin{equation}\notag
L=\left\{ \left(\begin{array}{ccc}
* & 0 & 0\\
0 & * & *\\
0 & * & *
\end{array}\right)\right\} 
\end{equation} 

and their commutator 
\begin{equation}\notag
[L,L]=\left\{ \left(\begin{array}{cc}
1 & 0\\
0 & A
\end{array}\right) : A\in SL(2,\mathbb{C})\right\}. 
\end{equation}

The subgroup $[L,L]$ acts by conjugation in both the space of lattices of $U$ and in that of $U^-$, more explicitly, if we take $F_1$ a lattice in $U$ and $F_2$ a lattice in $U^{-}$. Then the conjugation by an element $h=\left(\begin{array}{cc}
1 & 0\\
0 & A
\end{array}\right) \in H$ sends $F_1$ to $F_1A^{-1}$, $F_2$ to $AF_2$ and the pair $(F_1,F_2)$ to $(F_1A^{-1},AF_2)$. In addition, conjugation preserves the volume of the lattices.\\

Fix lattices $F_1$ and $F_2$, by the identification of $U$ and $U^-$ with $\mathbb{C}^2$ we can think the lattices $F_1$ and $F_2$ as lattices in $\mathbb{R}^4$. Denote by $\Omega_i$ the space of lattices in $\mathbb{R}^4$ of the same determinant as $F_i$. The subgroup  $H$ acts transitively on both $\Omega_1$ and $\Omega_2$, so we can identify $\Omega_i$ with $SL(4,\mathbb{R})/SL(4,\mathbb{Z})$.\\      

We can identify the subgroup $H$ with the subgroup of $SL(4,\mathbb{R})$ that centralizes the element $T_i=\left(\begin{array}{cccc}
0 & -1 & 0 & 0\\
1 & 0 & 0 & 0\\
0 & 0 & 0 & -1\\
0 & 0 & 1 & 0
\end{array}\right)$, denote by $\tilde{H}$ this subgroup.  Thus, if we take a representative $g_iSL(4,\mathbb{Z}) \in SL(4,\mathbb{R})$ of the lattice $F_i$, then $^{t}A^{-1}g_1SL(4,\mathbb{Z})$ and $Ag_2SL(4,\mathbb{Z})$ are the representatives of the lattices $hF_1h^{-1}$ and $hF_2h^{-1}$ respectively.  In this way the orbits $H.F_1$, $H.F_2$ and $H.(F_1,F_2)$ are identified with $\tilde{H} g_1SL(4,\mathbb{Z})$, $\tilde{H} g_2SL(4,\mathbb{Z})$ and $\{(^{t}A^{-1}g_1SL(4,\mathbb{Z}),Ag_2SL(4,\mathbb{Z})) : A \in \tilde{H}\}$. 

\begin{lemma}
Let $\{h_i\}$ be a sequence in $H$. If $<F_1,F_2>$ is discrete, then the set $M_1=\{h_iF_1h_i^{-1}\}$ is relatively compact if and only if $M_1=\{h_iF_2h_i^{-1}\}$
\end{lemma}

\begin{proof}
 Suppose that $M_1$ is relatively compact and $M_2$ is not. As $M_1$ is a relatively compact subset of the space of lattices in $U$, there exists a number $d>0$  such that every lattice in $M_1$ has a basis $\beta_i$, with $\left\|x-I_3\right\|<d$, for all $x \in \beta_i$. Denote by $\beta$ the union $\underset{i}{\bigsqcup}\beta_{i}$. Take $V$ an $\epsilon$-neighborhood of the identity, such that for every $x,y \in V$, we have   $\left\|[x,y]-I_3\right\|<\frac{1}{2}\min\{\left\|x-I_3\right\|,\left\|y-I_3\right\|\}$. Let $g$ in the centralizer of $H$, so that $\left\|gug^{-1}-I_3\right\|<\epsilon$ for all $u\in U$ with $\left\|u-I_3\right\|<d$. By the Mahler's compactness criterion, we can find an $\epsilon_0<\epsilon$ such that $\left\|z-I_3\right\| > \epsilon_0$ for all $z \in g_iM_1g_i^{-1}$, $z$ non trivial, in the other hand like $M_2$ is not relatively compact, there exists a $j \geq 1$, so that we can find a element $y \in gh_jF_2h_j^{-1}g^{-1}$ non trivial with $\left\|y-I_3\right\| \leq \epsilon_0$. \\

We define the sets $W_{-1}=\{gzg^{-1} : z \in U, z \in \beta \cap h_jF_1h_j^{-1}\}$, $W_0=\{y\}$ and $W_i=\{sts^{-1}t^{-1} : s \in W_{-1}, t \in W_{i-1}\}$. Like $W_{-1},W_0 \subset V$, then $\underset{w\in W_{i}}{\sup}\left\Vert w-I_{4}\right\Vert \rightarrow0$ as $i \rightarrow \infty$. In addition, $W_{-1},W_0$ are contained in the discrete subgroup $gh_j<F_1,F_2>h_j^{-1}g^{-1}$, so each $W_i$ is in $gh_j<F_1,F_2>h_j^{-1}g^{-1}$. Then, we can find a $k>0$ such that $U_k \neq \{I_3\}$ and $U_{k+1}=\{I_3\}$. Take $z \in U_k $, $z$ non trivial, with $\left\|z-I_3\right\| \leq \epsilon_0/3$. Note that $z$ belongs to $C(W_{-1})$, like the centralizer is algebraic and the Zariski closure of $\beta$ is $U$, then $z$ commutes with $gh_jUh_j^{-1}g^{-1}$. So, $h_j^{-1}g^{-1}zgh_j \in C(U) \cap <F_1,F_2>$, the centralizer of $U$ is contained in $Z(G) \dot U$, then $(h_j^{-1}g^{-1}zgh_j)^3 \in U \cap <F_1,F_2>$, thus  $h_j^{-1}g^{-1}z^3gh_j \in F_1$.   Then, we have $z^{3} \in gh_jF_1h_j^{-1}g^{-1} $ and $\left\|z-I_3\right\|< \epsilon_0$, that is a contradiction.         
\end{proof}

\begin{lemma}
Let $E_1,E_2$ be a pair of lattices in $\overline{H.(F_1,F_2)}$. If $<F_1,F_2>$ is discrete then the subgroup $<E_1,E_2>$ is discrete. 
\end{lemma}

\begin{proof}
 Take $\{h_i\}$ a sequence in $H$ such that $(h_iF_1h_i^{-1},h_iF_2h_i^{-1})$ converges to $(E_1,E_2)$. The set $A=\{h_iF_1h_i^{-1},h_iF_2h_i^{-1},E_1,E_2\}$ is relatively compact in the space of lattices, then there are constant $d_1,d_2>0$ such that if $F$ is a lattice in $A$, we can find a basis $\beta$ for $A$ with $d_1 < \left\|x-I_3\right\| <d_2$, for all $x \in \beta$. We define the set $K=\{g\in SL(3,\mathbb{C}) : d_1<\left\|g-I_3\right\|<d_2 \}$, note that for each $h_i$, the set $h_i<F_1,F_2>h_i^{-1} \cap K$ generates $h_i<F_1,F_2>h_i^{-1}$.  The subgroup $h_i<F_1,F_2>h_i^{-1}$ is Zariski dense and therefore has the property (P), in addition $K$ is compact and we can apply the Wang's theorem to find a neighborhood $V$ of the identity such that $V \cap h_i<F_1,F_2>h_i^{-1}=\{I_3\}$ for all $i$. Thus, by the Chabauty's theorem $<E_1,E_2>$ is discrete.  
\end{proof}

By the two previous lemmas, we have: 

\begin{proposition}\label{prop1}
Let $E_1$ a lattice in $\overline{H.F_1}$. If $<F_1,F_2>$ is discrete then there is a lattice $E_2$ in $U^{-}$ such that $(E_1,E_2)$ belongs to $\overline{H.(F_1,F_2)}$ and $<E_1,E_2>$ is discrete. 
\end{proposition}

Denote by $\phi$ the map that identify $SL(2,\mathbb{C})$ in $SL(4,\mathbb{R})$ before mentioned, explicitly $\phi$ is given by: 

\begin{equation}\notag
\phi\left(\begin{array}{cc}
a & b\\
c & d
\end{array}\right)=\left(\begin{array}{cccc}
\mathrm{re}(a) & -\mathrm{im}(a) & \mathrm{re}(b) & -\mathrm{im}(b)\\
\mathrm{im}(a) & \mathrm{re}(a) & \mathrm{im}(b) & \mathrm{re}(b)\\
\mathrm{re}(c) & -\mathrm{im}(c) & \mathrm{re}(d) & -\mathrm{im}(d)\\
\mathrm{im}(c) & \mathrm{re}(c) & \mathrm{im}(d) & \mathrm{re}(d)
\end{array}\right).
\end{equation} 

Applying the theorem \ref{Max} we obtain that $\phi(SL(2,\mathbb{C}))$ is a maximal connected subgroup of $SL(4,\mathbb{R})$, since $\phi$ is a non-trivial irreducible linear representation of $SL(2,\mathbb{C})$ that does not leave invariant any non degenerate bilinear form in $\mathbb{C}^4$. This fact reduces the possibilities for the closure of the orbits $H.F_1$ and $H.F_2$.   

\begin{proposition}
The orbits $H.F_1$ and $H.F_2$ are closed. 
\end{proposition}

\begin{proof}
 We give the proof for $H.F_1$, the orbit $H.F_2$ is treated analogously. \\
 Remember that we can identify the orbit $H.F_1$ with the orbit $\tilde{H}g_1SL(4,\mathbb{Z})$. So, we can apply the theorem \ref{ra} in the context of the homogeneous space $SL(4,\mathbb{R})/SL(4,\mathbb{Z})$ for the subgroup $\tilde{H}$. Then if the orbit $\tilde{H}g_1SL(4,\mathbb{Z})$ is not closed, its closure is \newline $SL(4,\mathbb{R})/SL(4,\mathbb{Z})$.  \\
 Take $V$ an $\epsilon-$neighborhood of Zassenhaus of the identity in $SL(6,\mathbb{R})$, that is, $V$ is such that for any discrete subgroup $\Gamma$ of $SL(6,\mathbb{R})$, $\Gamma \cap V$ is contained in a connected nilpotent Lie subgroup of $SL(6,\mathbb{R})$.  Also, take a $c>0$, given by the Hermite's theorem, such that any lattice $\Gamma$ in $\mathbb{R}^4$ has a non-zero element $x$ with $\left\|x\right\|< c \cdot \mathrm{det}(\Gamma)$. \\  
For an element 
$h=\left(\begin{array}{cc}
I_{2} & 0\\
y & I_{4}
\end{array}\right)$  with $\left\|y\right\|< c \cdot \mathrm{det} (F_2)$, we can find an $\epsilon_0>0$ small enough such that $\epsilon_0 <\epsilon$ and $\left\|[g,h]-I_6\right\|<\epsilon$ for all $g\in SL(6,\mathbb{R})$ in the neighborhood $\left\|g-I_6\right\|< \epsilon_0$. As the closure of the orbit $H.F_1$ is identifying with the whole space $SL(4,\mathbb{R})/SL(4,\mathbb{Z})$, it's possible to obtain a lattice $F_1'$ identified with $g_1^{'}SL(4,\mathbb{Z})$ with the same determinant that $g_1SL(4,\mathbb{Z})$ and that contains a set $S$ of elements of norm less than $\frac{\epsilon_0}{\sqrt{2}}$. Take the set: 

\begin{equation}\notag
\tilde{S}=\{\left(\begin{array}{cccccc}
1 & 0 & s_{1} & s_{2} & s_{3} & s_{4}\\
0 & 1 & -s_{2} & s_{1} & -s_{4} & s_{3}\\
0 & 0 & 1 & 0 & 0 & 0\\
0 & 0 & 0 & 1 & 0 & 0\\
0 & 0 & 0 & 0 & 1 & 0\\
0 & 0 & 0 & 0 & 0 & 1
\end{array}\right)
 : (s_1,s_2,s_3,s_4) \in S\}.
\end{equation}       

By the proposition \ref{prop1}, there exists a lattice $F_2'$ in $U^{-}$ identified with $g_2'SL(4,\mathbb{Z})$  of the same determinant of $F_2$, such that $<F_1',F_2'>$ is discrete, then take a element $y$ in $g_2'SL(4,\mathbb{Z})$ with $\left\|y\right\|< c \cdot \mathrm{det} (F_2)$ and $h=\left(\begin{array}{cc}
I_{2} & 0\\
y & I_{4}
\end{array}\right)$. All the elements in the set $A=\{s,[s,h] : s \in \tilde{S} \}$ are in $V$, so the subgroup generated by $A$ is nilpotent and thus the subgroup generated by $\{A,h\}$ is unipotent, this implies that $y$ is trivial, but this contradicts the Hermite's theorem.          

\end{proof}

Denote by $\Lambda_{F_i}=\{h \in H : hF_1h^{-1}=F_1\}$ the stabilizer of $F_i$ under the conjugation action by $H$. By the previous proposition $H.F_i$ is closed, so applying the Ratner's theorem we have that $\Lambda_{F_i}$ is a lattice in $H$. In other words, $\widetilde{\Lambda_{F_i}}=\{\tilde{h} \in \tilde{H} : \tilde{h} g_iSL(4,\mathbb{Z}) \tilde{h}^{-1}=F_1\}$ is a lattice in $\tilde{H}$, where $g_iSL(4,\mathbb{Z})$ is a representative of the lattice $F_i$. We have that $\widetilde{\Lambda_{F_i}} \subset g_iSL(4,\mathbb{Z})g_i^{-1}$ then $g_i^{-1}\widetilde{\Lambda_{F_i}}g_i \subset SL(4,\mathbb{Q})$, by the Borel's density theorem  $g_i^{-1}\widetilde{\Lambda_{F_i}}g_i$ is Zariski dense in $g_i^{-1} \widetilde{\mathbf{H}} g_i$, where $\widetilde{\mathbf{H}}$ is the centralizer of $T_i$ in $SL(4,\mathbb{C})$. As $\widetilde{\mathbf{H}}$ is a $\mathbb{Q}-$subgroup of $SL(4,\mathbb{C})$, we obtain a $\mathbb{Q}-$form for $\tilde{H}$ such that $\tilde{H}(\mathbb{Z})$ is commensurable to $\widetilde{\Lambda_{F_i}}$. This implies that $\Lambda_{F_i}$ is commensurable to an arithmetic subgroup $H(\mathbb{Z})$ for some $\mathbb{Q}-$form of $H$. \\

For a proof of the following theorem see \cite{RefMRT} or \cite{RefMORR}. 
\begin{theorem}
Let $H(\mathbb{Q})$ be a $\mathbb{Q}$-form of the group $H=SL(2,\mathbb{C})$. We have that $H(\mathbb{Q})$ is isomorphic to $SL(1,D)$ where $D$ is a quaternion algebra over a number field $k$ with exactly one complex place, such that $D$ is ramified at all real places.    
\end{theorem}

Let $D$ be a quaternion algebra as in the previous theorem. Consider $\mathcal{O}$ be an $R_k$-order of $D$, then $\Lambda_{F_i}$ is isomorphic to a subgroup commensurable with $\rho(SL(1,\mathcal{O}))$, where $R_k$ is the ring of integers of $k$ and $\rho$ is a $k-$embedding of $D$ into $M(2,\mathbb{C})$.  Then there exists an isomorphism $\psi$ such that $\psi(F_i)$ is commensurable to $\mathcal{O}$.  \\

Now, we proceed to study the double orbit $H.(F_1,F_2)$.

\begin{theorem}
The orbit $H.(F_1,F_2)$ is closed. 
\end{theorem}  

\begin{proof}
We identify the double orbit $H.(F_1,F_2)$ with the set \newline $\tilde{H}.(F_1,F_2)=\{(^{t}\phi(A)^{-1}g_1SL(4,\mathbb{Z}),\phi(A)g_2SL(4,\mathbb{Z})) : A \in SL(2,\mathbb{C})\}$, where $g_iSL(4,\mathbb{Z})$ is a representative of the lattice $F_i$ in $SL(4,\mathbb{R})/SL(4,\mathbb{Z})$. If $\tilde{H}.(F_1,F_2)$ is not closed, then   its closure  is $\tilde{H}.F_1 \times \tilde{H}.F_2$ as a consequence of the Ratner's theorem, where $\tilde{H}.F_1=\{\tilde{A}g_iSL(4,\mathbb{Z}) : \tilde{A} \in \tilde{H}\}$. \\
 Take an element $(\tilde{A}g_1SL(4,\mathbb{Z}),g_2SL(4,\mathbb{Z}))$ in $\tilde{H}.F_1 \times \tilde{H}.F_2$ with $\tilde{A}$ not in $SL(4,\mathbb{Z})$, if a sequence in $\{(^{t}\phi(A_i)^{-1}g_1SL(4,\mathbb{Z}),\phi(A_i)g_2SL(4,\mathbb{Z}))\}$ in $\tilde{H}.(F_1,F_2)$ converges to $(\tilde{A}g_1SL(4,\mathbb{Z}),g_2SL(4,\mathbb{Z}))$, then $\phi(A_i)g_2SL(4,\mathbb{Z})$ converges to $g_2SL(4,\mathbb{Z})$ and thus $^{t}\phi(A_i)^{-1}g_1SL(4,\mathbb{Z})$ converges to $g_1SL(4,\mathbb{Z})$, which is impossible. Thus $H.(F_1,F_2)$ is closed.  
\end{proof}

As $H.(F_1,F_2)$ is closed, then by the Ratner's theorem we have that $\Lambda_{F_1,F_2}$ is a lattice in $H$. Where $\Lambda_{F_1,F_2}=\{h \in H : (hF_1h^{-1},hF_2h^{-1})=(F_1,F_2)\}$ is the stabilizer of the pair of lattices $(F_1,F_2)$.  In other words, $\widetilde{\Lambda}_{F_1,F_2}=\{A \in \tilde{H} : (^{t}A^{-1} g_1SL(4,\mathbb{Z}),Ag_2SL(4,\mathbb{Z})) =(g_1SL(4,\mathbb{Z}),g_2SL(4,\mathbb{Z}))\}$ is a lattice in $\tilde{H}$, where $g_iSL(4,\mathbb{Z})$ is a representative of the lattice $F_i$. For similar reasons that in the case of single orbits, we obtain a $\mathbb{Q}-$form for $\tilde{H}$ such that $\tilde{H}(\mathbb{Z})$ is commensurable to $\widetilde{\Lambda}_{F_1,F_2}$. This implies that $\Lambda_{F_1,F_2}$ is commensurable to an arithmetic subgroup $H(\mathbb{Z})$ for some $\mathbb{Q}-$form of $H$. \\ 

\begin{proof}(of the arithmeticity of $<F_1,F_2>$) By the closeness of the single orbits we obtain $\mathbb{Q}-$forms of $U$ and $U^{-}$, and $\mathbb{Q}-$forms of their normalizers $P=N(U)$ and $P^{-}=N(U^{-})$. The normalizer of $<F_1,F_2> \cap H$ is discrete and contains $\Lambda_{F_1,F_2}$, as  $\Lambda_{F_1,F_2}$ is a lattice in $H$ then $N(<F_1,F_2> \cap H)$ is a lattice in $H$, then there exists a $\mathbb{Q}-$form of $G$ such that $N(<F_1,F_2> \cap H) \subset G(\mathbb{Q})$ and $P(\mathbb{Q})$, $P^{-}(\mathbb{Q})$ correspond to the $\mathbb{Q}-$forms obtained before. In particular $<F_1,F_2> \subset G(\mathbb{Q})$, we can apply the Theorem \ref{rv} to obtain that $<F_1,F_2>$ is an arithmetic subgroup of $SL(3,\mathbb{C})$.       
\end{proof}

\backmatter





\bmhead{Acknowledgments}

I thank to A. Cano for all the useful conversations and for his great support.


\section*{Statements and Declarations}

\bmhead{Funding}
No funding was received to assist with the preparation of this manuscript.

\bmhead{Conflict of interest/Competing interests}

\textbf{Financial interests:} The author declares that he has no financial interest. 

\textbf{Non-financial interests:} None.

\bmhead{Ethics approval} Not applicable
\bmhead{Consent to participate} Not applicable
\bmhead{Consent for publication} Not applicable
\bmhead{Availability of data and materials} Not applicable
\bmhead{Code availability} Not applicable
\bmhead{Authors' contributions} Not applicable

\end{document}